\documentclass[a4paper,10pt]{article}

\usepackage[utf8]{inputenc}
\usepackage{authblk}
\usepackage[english]{babel}
\usepackage{amsmath}
\usepackage{amsfonts}
\usepackage{amssymb}
\usepackage{amsthm}
\usepackage{mathrsfs} % for mathscr
\usepackage{verbatim} % for multiple-lines comments ("comment" environment)
\usepackage{etoolbox} % to use ifs
\usepackage[stretch=5,shrink=5]{microtype} % marvellous! But too invasive, defalut values would be 20 and 20.
\usepackage{hyperref} % clickable links
\usepackage{mathtools} % for \ceil and \floor and \rfloor
\usepackage[justification=centering]{caption}% centered caption
\usepackage{array} % for vertically centered cells in tables. (see https://tex.stackexchange.com/questions/7208/how-to-vertically-center-the-text-of-the-cells )
\usepackage{multirow} % for multiple-line spanning table fields
\usepackage{color}
\usepackage{float}
\usepackage{setspace}
\newcolumntype{M}[1]{>{\centering\arraybackslash}m{#1}}
\newcolumntype{N}{@{}m{0pt}@{}}

\newtheorem{theorem}{Theorem}  
\newtheorem{proposition}{Proposition} 
\newtheorem{remark}{Remark} 
\newtheorem{lemma}{Lemma} %

\makeatletter

\newcommand*\wthelper[2]{%
        \hbox{\dimen@\accentfontxheight#1%
                \accentfontxheight#11.2\dimen@ % change the scaling factor 1.2 to change the height of the tilde
                $\m@th#1\widetilde{#2}$%
                \accentfontxheight#1\dimen@
        }%
}
\newcommand*\accentfontxheight[1]{%
        \fontdimen5\ifx#1\displaystyle
                \textfont
        \else\ifx#1\textstyle
                \textfont
        \else\ifx#1\scriptstyle
                \scriptfont
        \else
                \scriptscriptfont
        \fi\fi\fi3
}
\makeatother
\newcommand{\randomchoose}[1]{We choose randomly the element\ifstrequal{#1}{s}{}{s} % is there more than a single random element?
}
\newcommand{\FF}{\mathbb{F}}
\newcommand{\Sym}{\mathrm{Sym}}
\newcommand{\Alt}{\mathrm{Alt}}

\newcommand{\GL}{\mathrm{GL}}

\newcommand{\N}{\text{$\mathbf{N}$}}
\newcommand{\MM}{\mathcal{M}}

\title{A note on an infeasible linearization of some block ciphers}

\author[1]{Riccardo Aragona\thanks{ric.aragona@gmail.com}}
\author[1]{Anna Rimoldi\thanks{anna.rimoldi@gmail.com}}
\author[1]{Massimiliano Sala\thanks{maxsalacodes@gmail.com }}

\affil[1]{Department of Mathematics, University of Trento, Italy}
\date{}

\begin{document}
	\maketitle
\begin{abstract}
A block cipher can be easily broken if its encryption functions can be seen as  
linear maps on a small vector space. Even more so, if its round functions can be seen as  
linear maps on a small vector space.
We show that this cannot happen for the AES. More precisely, we prove 
that if  the AES round transformations can be embedded into a linear cipher acting on a vector space, then this space is huge-dimensional  and so this embedding is infeasible in
practice.  We present two elementary proofs.
\end{abstract}

%\begin{keyword}
\small{\textbf{Keywords:} AES, block cipher, group theory.
%\end{keyword}
%\end{frontmatter}
%==================================================================

\section{Introduction}
The Advanced Encryption Standard (AES)  \cite{CGC-misc-nistAESfips197} is nowadays the most
widespread block cipher in commercial applications. It represents the state-of-the-art in block cipher design and the only known attack on its full version is the biclique attack \cite{bogdanov2011biclique}, which still requires an amount of cryptanalytic effort slightly less than the brute-force key-search. Practical attacks on reduced versions of the AES only tackle up to 6 rounds, with the Partial Sum attack being the most dangerous \cite{ferguson2001improved,partialnostro}
%It represents the state-of-art in block 
%cipher design and provides an unparalleled level of assurance
%against all known cryptanalytic techniques, except for its round-reduced versions (see for example \cite{daemen1997block,ferguson2001improved}).
%AES (and other modern block ciphers) presents a highly
%algebraic structure leading researchers to exploit it for new algebraic attacks,
%but these tries have been unsuccessful as yet, except for academic reduced
%versions.
The best that a designer can   hope for a block cipher is that all its encryption
functions behave in an unpredictable way (close to random), in particular the designer
aims at a cipher behaviour  totally different from linear or affine maps.\\
An indication of the cryptographic strength of the AES is that nobody has been able to show that
its encryption functions are any closer to linear maps than arbitrary random
functions. However, it might be possible to extend the AES to act on larger spaces, in
such a way that the possible non-random behaviour of the AES becomes easier to spot. Generally
speaking, the worst scenario for a designer consists of a space large enough to make the AES
linear but small enough to allow practical computations.

In this note we prove in two elementary ways, respectively using counting arguments (Section \ref{elements}) and number theory arguments (Section \ref{orders}),
that the round functions of the AES cipher cannot be embedded  into a linear group acting on a vector space $W$, unless the dimension of
$W$ is huge, making this embedding useless in
practice. Both proofs show that the smallest degree of a (faithful) 
representation of the group generated by the round functions 
of the AES, that is $\Alt((\FF_2)^{128})$ \cite{CGC-cry-art-sparr08}, is at least $2^{67}$. Since  computing a  $2^{67}\times 2^{67}$ matrix is infeasible in practice, our result shows that this attack cannot be mounted in practice. 

In 1976 Wagner \cite{wagner1976faithful} studied the  (faithful) linear representations  of
 $\Alt((\FF_2)^{128})$  and was already able to prove that their minimal degree is $2^{128}$. Therefore, we do not claim any new algebraic result in this note. The interest  of our note lies in three facts. First, we provide elementary proofs of our degree estimate, lacking a deep algebraic background, while Wagner's proof follows an involved argument relying on representation theory. Second, we show the link between a purely group theory result and a possible practical application in cryptanalysis, providing thus more assurance in the cipher itself (since the attack cannot be practically mounted). Third, our two proofs are based on utterly different arguments, one counting the number of group elements and the other estimating the maximal element order, but do lead to the same numerical estimate, which we find unexpected and deserves noting.
 
 This note is part of the PhD thesis of the second author \cite{thesisrimoldi}, which was never before published in a peer-reviewed journal.

%A classical proof
%that the least degree of a faithful representation of $\Alt(\FF^{128})$ is 128 is given in \cite{wagner1976faithful}. In \cite{wagner1976faithful} the author uses a representation theoretical  approach, here we obtain our result with an ``elementary'' combinatorial proof. Moreover our estimate is weaker than
%Wagner's, but strong enough to show the linearization infeasibility. 

\section{Preliminaries}
Let $n \geq 2$ be an integer. Let $\FF= \FF_2$ be the field with $2$ elements. Let $V=\FF^n$ be the vector space over $\FF$ of dimension $n$.
We denote by $\Sym(V)$ and $\Alt(V)$, respectively, the symmetric and 
alternating group on $V$.
We denote by $\GL(V)$ the group of all linear permutations of $V$.

%\\
%By taking $V=2^{|G|}$ we can always linearize $G$ over $V$ via the so-called {\em regular} representation, but of course this is huge and normally impractical.
%\\

Let ${\mathcal C}$ be any block cipher such that the plaintext space $\MM$
coincides with the ciphertext space. Let ${\mathcal K}$ be the key space.
Any key $k\in {\mathcal K}$ induces a permutation $\tau_k$ on $\MM$.
Since $\MM$ is usually $V=\FF^n$ for some $n\in \N$, we can view
$\tau_k$ as an element of $\Sym(V)$.
We denote by $\Gamma=\Gamma({\mathcal C})$ the subgroup of $\Sym(V)$ generated
by all the $\tau_k$'s.
Unfortunately, the knowledge of $\Gamma({\mathcal C})$ is out of reach
for the most important block ciphers, such as 
the AES \cite{CGC-misc-nistAESfips197} and the 
DES \cite{CGC-misc-nistDESfips46}. 
However, researchers have been able to compute another related
group. Suppose that ${\mathcal C}$ is the composition of $l$ rounds
(the division into rounds is provided in the document describing the cipher).
Then any key $k$ would induce $l$ permutations, 
$\tau_{k,1},\ldots,\tau_{k,l}$, whose composition is $\tau_k$.
For any round $h$, we can consider $\Gamma_h({\mathcal C})$ as the 
subgroup of $\Sym(V)$ generated by the $\tau_{k,h}$'s 
(with $k$ varying in ${\mathcal K}$).
We can thus define the group $\Gamma_{\infty}=\Gamma_{\infty}({\mathcal C})$
as the subgroup of $\Sym(V)$ generated by all the $\Gamma_h$'s.
Obviously, $\Gamma \,\leq\, \Gamma_{\infty} \,.$
The group $\Gamma_{\infty}$ is traditionally called the \emph{group generated
by the round functions} with  independent sub-keys.
This group is known for some important ciphers, for the AES we have

\begin{proposition}[\cite{CGC-cry-art-sparr08}]
\label{wer}

 $$\Gamma_{\infty}({\mathrm{AES}})=\Alt(\FF^{128}).$$
\end{proposition}

\begin{remark}
The fact that $\Gamma_{\infty}$ is the alternating group is not an exception holding only for the AES. Indeed, any block cipher built choosing accurately the cipher components (SBoxes and linear mixing layer) will have $\Gamma_{\infty}$ as either the  alternating group or the symmetric group, even if the cipher is defined over a positive characteristic (see \cite{caranti2009application,aragona2014group}).\\
\end{remark}
%
%It is common belief among researchers that $\Gamma_{AES}=\Gamma_{\infty}({\mathrm{AES}})=\Alt(\FF^{2^{128}})$.

%
Given a finite group $G$, we say that $G$ can be \emph{linearized} if there is an injective morphism $\pi:G \rightarrow \GL(W)$, for some vector space $W$
(this is called a ``faithful representation'' in representation theory).
If $G$ can be linearized, then, for any element $g \in G$, an attacker can compute a matrix $M_g$ corresponding to the action of $g$ over $W$ (via $\pi$). If the dimension of $W$ is sufficiently small, the matrix computation is straightforward, since it is enough to evaluate $g$ on a basis of $W$. In cryptanalysis, this attack would be called a \emph{chosen-plaintext attack} and can easily be translated into a \emph{known-plaintext attack} by collecting enough random  plaintext-ciphertext pairs.

In this note we show  that it is impossible to view $\Gamma_{\infty}({\mathrm{AES}})$  as a subgroup
of $\GL(W )$ with $W$ of small dimension. In
Cryptography it is customary to present estimates as powers of two, so our problem becomes to find the smallest $m$ such that
$\Gamma_{\infty}({\mathrm{AES}})$ can be linearized in $\GL(\FF^{2^m})$. 
%A classical proof is given in \cite{wagner1976faithful} where the author proves that $m = 128$. Here we obtain a result with an ``elementary'' combinatorial proof. Our estimate is weaker than
%Wagner's, but strong enough to show a linearization of AES is infeasible. 

There are two elementary ways to show that a finite group $H$ cannot be
contained (as isomorphic image) in a finite group $G$. The first is to show that
$|H| > |G|$, the second is to show that there is $\eta\in H$ such that its order is
strictly larger than the maximum element order in $G$. In Section \ref{elements} we present
our result using the first approach and we show that $m\geq 67$. In Section \ref{orders} we present our result using
the second approach and we show again that $m\geq 67$.

\section{Counting the group size}\label{elements}
%It is common belief among researchers that $\Gamma_{AES}=\Gamma_{\infty}({\mathrm{AES}})=\Alt(\FF^{128})$.
%Assuming this, we show in this section that it is impossible to view $\Gamma_{AES}$ as a subgroup of $\mathrm{GL}(V)$ with $V$ of small dimension. 
In this section we show that the order of $\Alt(\FF^{128})$ is strictly larger
than the order of $\GL(\FF^{2^{66}})$, so that if $\Alt(\FF^{128})<\GL(\FF^{2^{m}})$ then $m\geq 67$.\\
First we recall some well-known formulae: if $V=\FF^n$,
$$
  |\Sym(V)|= 2^{n}!, \quad |\Alt(V)|=\frac{2^{n}!}{2} \quad
  |\GL(V)|= \prod_{h=0}^{n-1}(2^{n}-2^h)<2^{n^2} \,.
$$
We begin with showing a lemma.

\begin{lemma}
\label{lemma-final}
The following inequality holds
$$ 2^{(2^7)^{19}}< 2^{128}!< 2^{(2^7)^{20}}.$$
\begin{proof}
Let $n=2^7$, we have to show $ 2^{n^{19}}< 2^n!< 2^{n^{20}}.$\\
We first show that $ 2^{n^{19}}< 2^n!$\,.\\
  %for $n \geq 2^7$:
Note that the following inequality 
\begin{equation}
\label{eqw}
\frac{1}{2^{n-i}}\geq \frac{1}{2^{n-i+1}-h}
\end{equation}
holds for $1\leq i \leq n-2$ and $ 1 \leq h \leq 2^{n-i}$.\\
Clearly
\begin{eqnarray*}
2^{n}! >2^{n^{19}} & \iff & 2^{n}(2^{n}-1)!> 2^{n} \cdot 2^{n^{19}-n}\\
                   & \iff & (2^{n}-1)(2^n-2)! > 2^{n^{19}-n}\cdot \frac{2^{n}-1}{2^{n}-1} \quad .
\end{eqnarray*}
We apply (\ref{eqw}) with $i=1$ and $h=1$ and so we must prove
$$(2^{n}-1)(2^n-2)! >2^{n^{19}-n}\cdot \frac{2^{n}-1}{2^{n-1}},$$
i.e.  
\begin{equation}
\label{eqw1}
(2^n-2)! >  2^{n^{19}-n-(n-1)}.
\end{equation}
In a similar way for $i=1$ and $h=2$, from (\ref{eqw1}) we obtain 
$$
(2^n-2)(2^n-3)! >  2^{n^{19}-n-(n-1)}\cdot \frac{2^{n}-2}{2^{n-1}},
$$
hence
$$
2^{n}! >2^{n^{19}} \iff (2^n-3)! >  2^{n^{19}-n-2(n-1)},
$$
and so on for all $3 \leq h \leq 2^{n-1}$ 
we obtain that we must verify
$$
(2^{n-1}-1)! >2^{n^{19}-n-2^{n-1}(n-1)}.
$$
Then we proceed by applying (\ref{eqw}) for all $2 \leq i \leq n-2$ 
and all $ 1 \leq h \leq 2^{n-i}$, so that we need only to prove
 $$
  (2^{n-(n-1)}-1)!\geq  2^{n^{19}-n-\sum_{i=1}^{n-1}2^{n-i}(n-i)} \,.
 $$
%
%\begin{eqnarray*}
%(2^n-2)!   &    >   & 2^{n^{19}-n-(n-1)}  \\
%           & \vdots &  \\
%(2^n-2^n+1)& \geq   & 2^{n^{19}-(2^n-1)n}
% \end{eqnarray*}
%
In other words, we have to prove  
\begin{equation}
\label{right-size}
1 > 2^{n^{19}-n-\sum_{i=1}^{n-1}2^{n-i}(n-i)}, \quad \textrm{that is,} \quad 0> n^{19}-n-\sum_{i=1}^{n-1}2^{n-i}(n-i).
\end{equation}
But a direct check shows that the right-hand size of (\ref{right-size}) 
holds when $n = 2^{7}$.\\ %Provato sperimentalmente.\\

 We are left with proving the following inequality: $2^n!< 2^{n^{20}}$.\\
  We proceed by induction for $2\leq n\leq 2^7$. In this range a computer computation shows that 
  \begin{equation}
  \label{b}
  n^{20} + 2^n n +2^n < (n+1)^{20}.
  \end{equation}
  When $n=2$, we have $2^2!< 2^{2^{20}}$.\\
 Suppose that $2^n!< 2^{n^{20}}$ and $n\leq 2^7$. 
We have to prove that $2^{(n+1)}!< 2^{(n+1)^{20}}$.
 Since $ 2^{n+1}! \,=\,  (2^n \cdot 2)!  = 2^n!(2^n+1) \cdots (2^n+2^n)$ and since $2^{n}+j\leq 2^{n+1}$ for all $1\leq j\leq 2^n$, 
we have
 $$
 2^n!(2^n+1) \cdots (2^n+2^n)  < 2^{n^{20}+n+1}\cdot (2^n +2) \cdots (2^n+2^n) 
         \leq \, 2^{n^{20}+2^n(n+1)} = 2^{n^{20}+2^n n +2^n}
 $$
 and, applying ($\ref{b}$), we get $2^{n^{20}+2^n n +2^n}<2^{(n+1)^{20}}.$
 %but  $n^{20} + 2^n n +2^n < (n+1)^{20}$ when $2 \leq n \leq 2^7$, so $2^{n^{20}+2^n n +2^n}<2^{(n+1)^{20}}.$\\
 Then the claimed inequality $2^{n+1}!< 2^{(n+1)^{20}}$ follows.
 \end{proof}
\end{lemma}

Our result is contained in the following proposition.
\begin{proposition}\label{prop1}
Let $W=\FF^{2^m}$ with $m \geq 2$. If $G < \mathrm{GL}(W)$, with $G$ isomorphic to $\Alt(\FF^{128})$,  then $m\geq 67$.
\begin{proof}
If $G<\GL(W)$, then $|G|\leq |\GL(W)|$. But $|\Sym(\FF^{128})|=2^{128}!>2^{2^{133}}$ thanks to Lemma \ref{lemma-final} and so 
$$
  |G|=|\Alt(\FF^{128})|= \frac{|\Sym(2^{128})|}{2} > \frac{2^{2^{133}}}{2}
  = 2^{2^{133}-1}> 2^{2^{132}}=2^{(2^{66})^2} > |\GL(\FF^{2^{66}})| \,.
$$
\end{proof}
\end{proposition}

\begin{remark}
We could improve the previous bound to $l\geq 68$ by using the finite version of the Stirling formula:
$$
n\log_2(n)-n\log_2(e) \leq \log_2(n!) \leq n\log_2(n) - n\log_2(e) + \log_2 (n),
$$
or equivalently
\begin{equation}\label{stirling}
\left(\frac{n}{e}\right)^n\leq n! \leq n \left(\frac{n}{e}\right)^n.
\end{equation}
However, our proof involves only elementary combinatorial arguments, which are not enough to prove \eqref{stirling}, since
the latter requires some non-algebraic arguments, such as mathematical analysis.
\end{remark}
%and so the result of Proposition \ref{prop1}
% could  be used also for any other block cipher ${\mathcal C}$ acting on $128-$bit messages such that 
%$\Gamma_{\infty}({\mathcal C})=\Alt(\FF^{128})$, e.g. the SERPENT 
%\cite{CGC-cry-art-serpent}.

%
%An immediate consequence of the previous proposition is that the AES cannot be 
%linearized without using matrices of size larger than $2^{66}$, 
%which is obviously impractical.
%An open question is to determine the smallest dimension $d$ of $V$ such that AES can be linearized. We note that the regular representation gives a bound $d \leq 2^{128}$, so $2^{66}<d \leq 2^{128}$. 
%We have some work in progress to strengthen our bound using more advanced group theory, that is, the classification of maximal subgroups of $\mathrm{GL}(V)$ (Aschbacher's theorem and its applications).

\section{Computing the maximum order of elements}\label{orders}

In this section we compare the maximum order of elements in the two groups $\Alt(\FF^{128})$ and $\GL(\FF^{2^m})$. We use permutations of even order. We denote by $\mathrm{o}(\sigma)$ the order of any permutation $\sigma$.\\
We need the following two theorems.

\begin{theorem} [\cite{darafsheh2008maximum}]\label{teo1} Let $\sigma\in\GL(\FF^N)$, with $\mathrm{o}(\sigma)$  even and $N\geq 4$.
Then
$$
\mathrm{o}(\sigma) \leq  2(2^{N-2}-1) = 2^{N-1}-2.
$$

Moreover, there is $\sigma\in\GL(\FF^N)$ whose order attains the upper bound.
\end{theorem}
\begin{proof} It follows directly from Theorem 1 in \cite{darafsheh2008maximum}, with $p = q = 2$ and
$N \geq 4$ (so point (a) and (b) do not apply).
\end{proof}

\begin{theorem}[Theorem 5.1.A at p. 145 in \cite{dixon1996permutation}]\label{teo2} Let $\nu \geq 3$ and $n = 2^\nu$. Then $\Alt(\FF^\nu)$ contains
an element $\eta$ of order (strictly) greater than $e^{\sqrt{(1/4)n \ln n)}}$.
\end{theorem}
In order to be able to compare the two estimates coming from Theorem \ref{teo1}
 and Theorem \ref{teo2}, we rewrite Theorem \ref{teo2} as follows, in order to have
$\mathrm{o}(\sigma)$ even. Our proof is an easy adaption of the proof contained in \cite{dixon1996permutation}.
\begin{theorem}\label{to3}
Let $\nu \geq 7$ and $n = 2^\nu$. Then $\Alt(\FF^\nu)$ contains
an element $\eta$ with $\mathrm{o}(\eta)>e^{\sqrt{(1/4)n \ln n)}}$ and $\mathrm{o}(\eta)$  even.
\end{theorem}
\begin{proof}
Let $z$ be a prime number such that $4+\sum_{3\leq p \leq z} p \leq n$, where the sum runs over distinct prime numbers from 3 to $z$.
Then $\Alt(\FF^\nu)$ contains an element $\eta_z$ whose non-trivial cycles are two transpositions and some cycles with length $3,\dots,z$. Hence
$\mathrm{o}(\eta_z)=2\prod_{3\leq p\leq z}p$.\\
If we show that there is a prime number $z$ such that 
$$
4+\sum_{3\leq p \leq z} p \leq n\quad\mbox{and}\quad\ln(\mathrm{o}(\eta_z))^2>\frac{1}{4}n\ln (n),
$$
we have done.\\
Let $\theta(z)=\ln(\mathrm{o}(\eta_z))=\ln (2)+\sum_{3\leq p \leq z} \ln (p)$ and let us denote $\theta^*(z)=\theta(z)-\ln (2)=\sum_{3\leq p \leq z} \ln (p)$.\\
Let $f(z) = \frac{z}{\ln(z)}$ . Since $f(z)$ is an increasing function for real $z > e$, in the case
when $z$ is a real number and $z \geq 19$ (for $z=19$ note  that $4 +
\sum_{2<p\leq 19} p = 79 < 128=2^7 \leq n$), we have that
$$
f(4) \ln(4) + f(3) \ln(3) = 7 < f(19) \ln(3) \leq f(z) \ln(3).
$$
So, if $z\geq 19$ and $z\in\mathbb{R}$, we can write
\begin{align*}
4 +\sum_{2< p\leq z} p &=f(4) \ln(4) +\sum_{2<p\leq z} f(p) \ln(p)\\
&=f(4) \ln(4) + f(3) \ln(3) +\sum_{3<p\leq z} f(p) \ln(p)\\
&< f(z) \ln(3) +\sum_{3<p\leq z} f(z) \ln(p)\\
&=\sum_{2<p\leq z} f(z) \ln(p)= f(z)
\sum_{2<p\leq z} \ln(p)= f(z)\theta^*(z).
\end{align*}
We shall choose $\bar{z} \geq 19$ such that $f(\bar{z})\theta^*(\bar{z}) = n$. Such a  $\bar{z}$ exists because
$f(19)\theta^*(19) < 100 < n$ and $f(z)\theta^*(z)$ is an increasing function assuming all
values.\\
Since $\theta^*(\bar{z}) > \bar{z}/2$ for  $\bar{z} \geq 19$, we have
$$
n=f(\bar{z})\theta^*(\bar{z})  =\frac{\bar{z}\theta^*(\bar{z})}{\ln(\bar{z})}<\frac{2(\theta^*(\bar{z}))^2}{\ln(2\theta^*(\bar{z}))}=
\frac{4(\theta^*(\bar{z}))^2}{2 \ln(2\theta^*(\bar{z}))}= f(4(\theta^*(\bar{z}))^2).
$$
However we also have $f(n \ln(n)) < n$. Since $f$ is an increasing function, this
shows that $n \ln(n) < 4(\theta^*(\bar{z}))^2 < 4(\theta(\bar{z}))^2$. It is now enough to consider $\tilde{z}$ as
the largest prime smaller than $\bar{z}$.
\end{proof}
Now, we compare the estimates from Theorem \ref{teo1} and Theorem \ref{teo2}.
Take $n = 2^{128}$ and $\eta \in\Alt(\FF^{128})$ such that $\mathrm{o}(\eta) \geq e^t$ ($\mathrm{o}(\eta)$ even), where
$t = \sqrt{(1/4)n \ln (n)} =\sqrt{(1/4)2^{128} \ln(2^{128})}$.
Since
$$
e^t = e^{\sqrt{2^{126}128 \ln (2)}} = e^{\sqrt{2^{133} \ln (2)}} = (e^{\sqrt{2 \ln (2)}})^{2^{66}},
$$
by replacing $e$ with $2^{\log_2 (e)}$, we obtain
$$
e^t = (2^{\log_2 (e)\sqrt{2 \ln (2)}})^{2^{66}}= 2^{2^{66}\log_2 (e)\sqrt{2 \ln (2)}} = 2^{2^{66}\epsilon},
$$
where $\epsilon\in\mathbb{R}$ is about  1.7. According to Theorem \ref{to3},  $\mathrm{o}(\eta) \geq e^{2^{66}\epsilon}$. 
If $\Alt(\FF^{128}) \subset \GL(\FF^N)$, we then need the the smallest
$N$ such that $\mathrm{o}(\eta) \leq (2^{N-1} -2)$ (Theorem \ref{teo1}). In other words we have to see
when the following inequality holds
\begin{equation}\label{oeta}
\mathrm{o}(\eta)= e^{2^{66}\epsilon} \leq (2^{N-1} -2).
\end{equation}
We observe that
\begin{itemize}
\item if $N = 2^{66}$, then \eqref{oeta} is false, since $2^{2^{66}\epsilon} > 2^{2^{66}} > 2^{2^{66}-1} - 2$;
\item if $N = 2^{67}$, then \eqref{oeta}  is true, since  $2^{2^{66}\epsilon} < 2^{2^{66}(1.7)} < 2^{2^{67}-1} - 2$.
\end{itemize}
Therefore, we need at least $m\geq 67$  to embed $\Alt(\FF^{128}) \subset \GL(\FF^{2^m})$, which is exactly the same value as in Proposition \ref{prop1}.

\section{Conclusion}
In this note we provided two elementary proofs,  lacking a deep algebraic background,  that the round functions of the AES cipher cannot be embedded  into a linear group acting on a vector space $W$, unless the dimension of
$W$ is at least $2^{67}$. Since  computing a  $2^{67}\times 2^{67}$ matrix is infeasible in practice, our result shows that this attack cannot be mounted in practice. Moreover, since we do not use  the specific structure of the AES in our proof of  Proposition \ref{prop1}, we note that such result could  be used also for any other block cipher ${\mathcal C}$ acting on $128-$bit messages such that 
the group generated by its round functions is $\Alt(\FF^{128})$, e.g.  SERPENT \cite{serpentalter}, or  $\Sym(\FF^{128})$.\\
Moreover we observe that for any block cipher ${\mathcal C}$ acting on $64-$bit messages such that 
$\Gamma_{\infty}({\mathcal C})=\Alt(\FF^{64})$ (e.g. KASUMI \cite{sparr2015round} and  a special extension of GOST \cite{GOSTarcarsal}) with a similar argument of Lemma \ref{lemma-final} we obtain
$$
(2^6)^{11}<2^{64}!<(2^6)^{13}.
$$
Hence
$$
|\Alt(\FF^{64})|= \frac{|\Sym(2^{64})|}{2} > \frac{2^{2^{66}}}{2}
  = 2^{2^{66}-1} > 2^{2^{64}}=2^{(2^{32})^2}> |\GL(\FF^{2^{32}})|
$$
and so if $\Alt(\FF^{64})<\GL(\FF^{2^{m}})$, then $m\geq 33$. Also in this case, we can conclude that the embedding of  GOST and KASUMI in a linear cipher is infeasible in practice.\\
Finally we note that the infeasibility of this type of linearization of a block cipher provides an additional motivation to check the size of the group generated by the round functions of a block cipher. In particular, it is important to check the optimal case when this group is the alternating or the symmetric group.

%%%%%%%%%%%%%%%%%%%%%%%%%%%%%%%%%%%%%%%%%%%%%%%%%%
\section*{Acknowledgment}
The authors would like to thank R\"udiger Sparr and Ralph Wernsdorf for their helpful suggestions and interesting comments.
%%%%%%%%%%%%%%%%%%%%%%%%%%%%%%%%%%%%%%%%%%%%%%%%%%%

\bibliographystyle{plain}
\bibliography{biblio}

\end{document}